\documentclass[a4paper]{amsart}

\usepackage{amsmath,amstext,amssymb,mathrsfs,amscd,amsthm,indentfirst}
\usepackage{amsfonts}
\usepackage{mathtools}
\usepackage{stmaryrd}

\usepackage{booktabs}
\usepackage{verbatim}
\usepackage{enumerate}

\usepackage{graphicx}

\numberwithin{equation}{section}

\newcommand{\CircNum}[1]{\ooalign{\hfil\raise .00ex\hbox{\scriptsize #1}\hfil\crcr\mathhexbox20D}}

\newcommand{\bF}{\mathbb{F}}

\newcommand{\bN}{\mathbb{N}}

\newcommand{\bQ}{\mathbb{Q}}
\newcommand{\bR}{\mathbb{R}}

\newcommand{\bZ}{\mathbb{Z}}

\newcommand\lra{\longrightarrow}

\newcommand\Diff{\mathrm{Diff}}

\newcommand\colim{\operatorname*{colim}}

\newcommand{\hcoker}{/\!\!/}

\newcommand{\Fr}{\mathrm{Fr}}
\newcommand{\R}{\bR}

\newcommand{\hAut}{\mathrm{hAut}}

\newcommand{\moddd}{/\!\!/}

\renewcommand{\epsilon}{\varepsilon}

\newcommand{\Gr}{\mathrm{Gr}}

\newcommand{\cM}{\mathcal{M}}

\newcommand{\GL}{\mathrm{GL}}

\newcommand{\Coeffs}{A} 

\mathchardef\ordinarycolon\mathcode`\:
\mathcode`\:=\string"8000
\begingroup \catcode`\:=\active
  \gdef:{\mathrel{\mathop\ordinarycolon}}
\endgroup

\usepackage[all]{xy}
\CompileMatrices

\usepackage{amsthm}
\theoremstyle{plain}

\newtheorem{theorem}{Theorem}[section]

\newtheorem{lemma}[theorem]{Lemma}

\theoremstyle{definition}

\theoremstyle{remark}

\newtheorem{remark}[theorem]{Remark}
\newtheorem*{remark*}{Remark}

\title{Operations on stable moduli spaces}

\author{S{\o}ren Galatius}
 \email{galatius@math.ku.dk}
 \address{Department of Mathematics\\
   University of Copenhagen\\
   Denmark}

\author{Oscar Randal-Williams}
\email{o.randal-williams@dpmms.cam.ac.uk}
\address{Centre for Mathematical Sciences\\
Wilberforce Road\\
Cambridge CB3 0WB\\
UK}

\begin{document}
\begin{abstract}
We construct certain operations on stable moduli spaces and use them to compare cohomology of  moduli spaces of closed manifolds with tangential structure.  We obtain isomorphisms in a stable range provided the $p$-adic valuation of the Euler characteristics agree, for all primes $p$ not invertible in the coefficients for cohomology.
\end{abstract}
\maketitle

\section{Introduction}

An influential theorem of Harer \cite{H} shows that the cohomology of the moduli stack $\mathcal{M}_g$ of genus $g$ Riemann surfaces is independent of $g$ in a range of degrees called the stable range, even though there is no direct map between the moduli spaces for different genera.  With rational coefficients the cohomology in the stable range is a polynomial ring, but with more general coefficients it is best described via \emph{infinite loop spaces}, as shown by \cite{Tillmann,MT,MW}.  In earlier papers (\cite{GR-W2,GR-W3,GR-W4}, see also \cite{GRWguide} for a survey) we have studied moduli spaces of higher dimensional manifolds, and in some cases have again shown that different moduli spaces have isomorphic cohomology in a range of degrees.  For $n > 1$ one can in most cases not make an integral comparison of moduli spaces of manifolds related by connected sum with copies of $S^n \times S^n$, at least not by an obvious generalization of the $n=1$ case, where a zig-zag of integral homology equivalences can be defined using manifolds with boundary.  In this paper we show that a comparison is possible after all, although not with all coefficient modules.  We also give examples showing that assumptions on the coefficients are necessary.

\subsection{Comparing moduli spaces of closed manifolds}

All manifolds in this paper will be smooth, compact, connected, and without boundary. If $W$ denotes such a manifold then there is a \emph{moduli space} $\cM(W)$ classifying smooth fibre bundles whose fibres are diffeomorphic to $W$.  As a model we may take $\cM(W) = B\Diff(W)$, the classifying space of the diffeomorphism group $\Diff(W)$ of $W$, equipped with the $C^\infty$ topology. Then for $A$ an abelian group $H^i(\cM(W);\Coeffs)$ is the group of $H^i(-;\Coeffs)$-valued characteristic classes of such fibre bundles.

Now let $d = 2n$ and $W$ be a $d$-manifold.  The connected sum $W \# (S^n \times S^n)$ is then well defined up to (non-canonical) diffeomorphism, as $S^n \times S^n$ admits an orientation-reversing diffeomorphism, and we write $W \# g(S^n \times S^n)$ for the $g$-fold iteration of this operation.  Two manifolds $W$ and $W'$ are called \emph{stably diffeomorphic} if $W \# g(S^n \times S^n)$ is diffeomorphic to $W' \# g'(S^n \times S^n)$ for some $g,g' \in \bN$.  For example, any two orientable connected surfaces are stably diffeomorphic, while two non-orientable connected surfaces are stably diffeomorphic if and only if their Euler characteristic have the same parity.

In this paper we shall ask about the relationship between $H^*(\cM(W);\Coeffs)$ and $H^*(\cM(W');\Coeffs)$ when $W$ and $W'$ are stably diffeomorphic. As a special case our main result will provide a canonical isomorphism 
$$H^i(\cM(W);\bZ_{(p)}) \cong H^i(\cM(W');\bZ_{(p)})$$
as long as these manifolds are simply-connected and of dimension $2n > 4$, and both $(-1)^n \chi(W)$ and $(-1)^n \chi(W')$ are large compared with $i$ and have the same $p$-adic valuation.

The precise statement of our main result applies more generally, and before giving it we first explain its natural setting. If $W$ is given an orientation $\lambda$ then there is a corresponding moduli space $\cM^\mathrm{or}(W,\lambda)$ classifying smooth fibre bundles with oriented fibres which are oriented diffeomorphic to $(W,\lambda)$, and a forgetful map $\cM^\mathrm{or}(W,\lambda) \to \cM(W)$. Then the connected sum $W \# g(S^n \times S^n)$ inherits an orientation, well defined up to oriented diffeomorphism, and we say that $(W,\lambda)$ is \emph{oriented stably diffeomorphic} to $(W',\lambda')$ provided $W \# g(S^n \times S^n)$ is oriented diffeomorphic to $W' \# g' (S^n \times S^n)$ for some $g,g' \in \bN$.  In this situation our result will also imply a canonical isomorphism $H^i(\cM^\mathrm{or}(W,\lambda);\bZ_{(p)}) \cong H^i(\cM^\mathrm{or}(W',\lambda');\bZ_{(p)})$, under the same hypotheses.

More generally, for a space $\Lambda$ equipped with a continuous action of $\GL_{d+1}(\R)$ a $\Lambda$-structure on a $d$-manifold $W$ is a $\GL_d(\R)$-equivariant map $\lambda: \Fr(TW) \to \Lambda$, or, equivalently, a $\GL_{d+1}(\R)$-equivariant map $\Fr(\epsilon^1 \oplus TW) \to \Lambda$.  For example, if $\Lambda = \{\pm 1\}$ on which $\GL_{d+1}(\R)$ acts by multiplication by the sign of the determinant, then a $\Lambda$-structure $\lambda: \Fr(TW) \to \{\pm 1\}$ is the same thing as an orientation: it distinguishes oriented frames from non-oriented ones.  Two $\Lambda$-structures on the same manifold are homotopic if they are homotopic through equivariant maps, and $(W,\lambda)$ is $\Lambda$-diffeomorphic to $(W',\lambda')$ if there exists a diffeomorphism $\phi: W \to W'$ such that $\lambda \circ D\phi$ is homotopic to $\lambda'$.  The usual embedding of $S^n \times S^n \subset \R^{2n+1}$ as the boundary of a thickened $S^n \times \{0\} \subset \R^{n+1} \times \R^n$ gives a trivialisation of $\epsilon^1 \oplus T(S^n \times S^n)$ and a $\Lambda$-structure on $W$ extends to one on $W \# (S^n \times S^n)$, canonically up to $\Lambda$-diffeomorphism.  For two pairs $(W,\lambda)$ and $(W',\lambda')$ consisting of a manifold and a $\Lambda$-structure, we say that they are \emph{stably $\Lambda$-diffeomorphic} if $W \# g(S^n \times S^n)$ is $\Lambda$-diffeomorphic to $W' \# g'(S^n \times S^n)$ for some $g,g' \in \bN$.

There is a moduli space $\cM^\Lambda(W,\lambda)$ parametrising smooth fibre bundles $\pi: E \to X$ with $d$-dimensional fibres, and where the fibrewise tangent bundle $T_\pi E$ is equipped with an equivariant map $\Fr(\epsilon^1 \oplus T_\pi E) \to \Lambda$, such that all fibres of $\pi$ are $\Lambda$-diffeomorphic to $(W,\lambda)$. Our main result is then as follows.

\begin{theorem}\label{mainthm:stability}
  Let $\Lambda$ be as above, and let $\lambda$ and $\lambda'$ be $\Lambda$-structures on $W$ and $W'$ such that $(W,\lambda)$ is stably $\Lambda$-diffeomorphic to $(W',\lambda')$.  For an abelian group $\Coeffs$ there is a canonical isomorphism
  \begin{equation*}
    H^i(\cM^\Lambda(W,\lambda);\Coeffs) \cong H^i(\cM^\Lambda(W',\lambda');\Coeffs),
  \end{equation*}
induced by a zig-zag of maps of spaces, provided
  \begin{enumerate}[(i)]
  \item $d = 2n > 4$ and $W$ and $W'$ are simply connected,
	
  \item the integers $(-1)^n \chi(W)$ and $(-1)^n\chi(W')$ are both $\geq 4i + C$, where
    \begin{equation*}
      C = 6 + \min \{(-1)^n \chi(W_0) \mid \text{$(W_0,\lambda_0)$ stably $\Lambda$-diffeomorphic to $(W,\lambda)$ and $(W',\lambda')$}\}.
    \end{equation*}

  \item $\chi(W)$ and $\chi(W')$ are both non-zero, and $v_p(\chi(W))=v_p(\chi(W'))$ for all primes $p$ which are not invertible in $\mathrm{End}_\bZ(\Coeffs)$.
  \end{enumerate}
\end{theorem}

In Section \ref{sec:example} we give an example showing the third condition cannot be relaxed.

The main results of \cite{GR-W2, GR-W3, GR-W4}, summarised in \cite{GRWguide}, provide a map
\begin{equation}\label{eq:Scan1}
\cM^\Lambda(W,\lambda) \lra (\Omega^\infty MT\Theta) \moddd \hAut(u)
\end{equation}
which is an isomorphism on homology in a range of degrees, when regarded as a map to the path component which it hits.  Similarly there is a map
\begin{equation}\label{eq:Scan2}
\cM^\Lambda(W',\lambda') \lra (\Omega^\infty MT\Theta) \moddd \hAut(u)
\end{equation}
which is an isomorphism on homology in a range of degrees, when regarded as a map to the path component which it hits. The definition of the codomains is recalled below.  However, if $\chi(W) \neq \chi(W')$ then these two maps land in different path components, and the problem becomes to compare the homology of these two path components.

\begin{remark}
Using the results of Friedrich \cite{Friedrich}, Theorem~\ref{mainthm:stability} can be extended to manifolds with virtually polycyclic fundamental groups. In this case the constant $C$ should be replaced by $C+4+ 2h$ where $h$ denotes the Hirsch length of the common fundamental group of $W$ and $W'$.
\end{remark}

\subsection{Operations on infinite loop spaces}
\label{sec:oper-infin-loop}

The data involved in defining the common target of the maps \eqref{eq:Scan1} and \eqref{eq:Scan2} is a $\GL_{2n}(\bR)$-equivariant fibration $u : \Theta \to \Lambda$ with domain which is cofibrant as a $\GL_{2n}(\bR)$-space. Letting $B$ denote the Borel construction $\Theta \moddd \GL_{2n}(\bR)$, $MT\Theta$ is then the Thom spectrum of the inverse of the canonical $2n$-dimensional vector bundle over $B$, and $\Omega^\infty MT\Theta$ is its associated infinite loop space. By functoriality the group-like topological monoid $\hAut(\Theta)$ of $\GL_{2n}(\bR)$-equivariant homotopy equivalences $f: \Theta \to \Theta$ acts on the infinite loop space $\Omega^\infty MT\Theta$, so the group-like submonoid $\hAut(u) = \{f \in \hAut(\Theta) \, | \, u \circ f = u\}$ does too. The target
$$(\Omega^\infty MT\Theta) \moddd \hAut(u)$$
of the maps \eqref{eq:Scan1} and \eqref{eq:Scan2} is the Borel construction for this action.

In order to prove Theorem \ref{mainthm:stability} we shall construct certain operations on the space
$\Omega^\infty MT\Theta$, in the case where the $\GL_{2n}(\bR)$-space $\Theta$ is obtained by restriction from a cofibrant $\GL_{2n+1}(\bR)$-space $\overline{\Theta}$. The space $\overline{B} = \overline{\Theta}  \hcoker \GL_{2n+1}(\bR)$ carries a canonical $(2n+1)$-dimensional vector bundle, and $MT\overline{\Theta}$ denotes its associated Thom spectrum; as above, by functoriality it carries an action of the monoid $\hAut(\overline{\Theta})$ of $\GL_{2n+1}(\bR)$-equivariant homotopy equivalences $f: \overline{\Theta} \to \overline{\Theta}$.

A key construction in this paper is a homotopy pullback diagram of infinite
loop spaces, equivariant for $\hAut(\overline{\Theta})$, of the form
\begin{equation}
  \label{eq:14}
  \begin{aligned}
  \xymatrix{
    \Omega^\infty MT\Theta \ar[r]\ar[d] & \Omega^{\infty-1}
    MT\overline{\Theta}\ar[d]\\
    Q(\overline{B}_+) \ar[r] & \Omega^\infty C_{st},
  }
  \end{aligned}
\end{equation}
whose bottom right corner has $\pi_0 \cong \bZ/2$ and all higher homotopy groups are $2$-power torsion, and the bottom horizontal map induces a surjection on $\pi_1$.  It induces an isomorphism
\begin{equation}\label{eq:13}
  \pi_0 MT\Theta \xrightarrow\cong \{(\chi,x) \in \bZ \times \pi_{-1}MT\overline{\Theta}
  \mid \chi \bmod 2 =  w_{2n}(x) \},
\end{equation}
whose first coordinate is given by the Euler class and whose second
coordinate is given by the stabilisation map.  To explain this claim and its notation, first note that the $2n$-dimensional vector bundle over $B$ has an Euler class $e \in H^{2n}(B; \bZ^{w_1})$, where the coefficients are twisted by the determinant of this vector bundle, and under the Thom isomorphism this gives a class $e \smile u_{-2n} \in H^0(MT\Theta;\bZ)$. Then $\chi$ is the value of this spectrum cohomology class on the Hurewicz image of an element of $\pi_0 MT\Theta$; geometrically, it assigns to such an element the Euler characteristic of a manifold representing it. Similarly, the $(2n+1)$-dimensional vector bundle over $\overline{B}$ has a $2n$th Stiefel--Whitney class $w_{2n} \in H^{2n}(\overline{B};\bZ/2)$, and under the Thom isomorphism this gives a class $w_{2n} \smile u_{-2n-1} \in H^{-1}(MT\overline{\Theta};\bZ/2)$. Then $w_{2n}(x)$ denotes the value of this spectrum cohomology class on the Hurewicz image of $x$.

\begin{theorem}\label{thm:scale-euler-char}
For $\chi \in \bZ$, write $\Omega^\infty_\chi MT\Theta$ for the inverse image of
  $\chi$ under the map $\Omega^\infty MT\Theta \to \bZ$ induced by the
  class $e \smile u_{-2n} \in H^0(MT\Theta; \bZ)$, i.e.\
  the union of the path components of the form $(\chi,?)$ under the
  bijection~\eqref{eq:13}.

  For any odd number $q$ there exists a self-map $MT\Theta \to MT\Theta$ inducing a map
  \begin{equation*}
    \psi^q: \Omega^\infty_\chi MT\Theta \lra \Omega^\infty_{q\chi} MT\Theta
  \end{equation*}
  such that
  \begin{enumerate}[(i)]
  \item\label{item:3} $\psi^q$ commutes (strictly) with the action of  $\hAut(\overline{\Theta})$,
  \item\label{item:1} $\psi^q$ is over the identity map of $\Omega^{\infty -1}
    MT\overline{\Theta}$,
  \item\label{item:2} $\psi^q$ induces an isomorphism in homology with coefficients in any $\bZ[q^{-1}]$-module.
  \end{enumerate}
\end{theorem}

  We shall also prove a version of Theorem~\ref{thm:scale-euler-char} for $q=2$, although it will be marginally weaker in that rather than the map $\psi^q$ being defined integrally and inducing an isomorphism with coefficients in any $\bZ[q^{-1}]$-module, the map $\psi^2$ will only be defined after localising the spaces involved away from $2$.
  
  \begin{theorem}\label{thm:scale-euler-char2}
  In the setup of Theorem \ref{thm:scale-euler-char}, if $\chi$ is
  even then there is a $\hAut(\overline{\Theta})$-equivariant weak equivalence of localised spaces
  \begin{equation*}
    \psi^2: (\Omega^\infty_\chi MT\Theta)[\tfrac12] \lra (\Omega^\infty_{2\chi}
    MT\Theta)[\tfrac12]
  \end{equation*}
  over the identity map of $(\Omega^{\infty-1} MT\overline{\Theta})[\frac12]$.
  \end{theorem}

  The operations in Theorems~\ref{thm:scale-euler-char} and~\ref{thm:scale-euler-char2} will arise from self-maps of the lower left corner in~\eqref{eq:14}.
  
  The proof of Theorem \ref{mainthm:stability} will use these operations to give endomorphisms of the space $ (\Omega^\infty MT\Theta) \moddd \hAut(u)$ which mix path-components, allowing us to compare the path components hit by the maps \eqref{eq:Scan1} and \eqref{eq:Scan2}. This strategy is analogous to arguments of Bendersky--Miller \cite{BenderskyMiller} and Cantero--Palmer \cite{CanteroPalmer} for cohomology of configuration spaces. This strategy has also been used by Krannich \cite{KrannichExotic} to show that  $H^i(\cM^{\mathrm{or}}(W, \lambda);\Coeffs) \cong H^i(\cM^{\mathrm{or}}(W \# \Sigma, \lambda);\Coeffs)$ for $(W, \lambda)$ an oriented manifold of dimension $2n > 4$ and $\Sigma$ an exotic sphere, in a stable range of degrees when the order of $[\Sigma] \in \Theta_{2n}$ is invertible in $\mathrm{End}_\bZ(\Coeffs)$.

\subsection*{Acknowledgements}  SG was supported by the European Research Council (ERC) under the European Union's Horizon 2020 research and innovation programme (grant agreement No.\ 682922), the Danish National Research Foundation through the Centre for Symmetry and Deformation (DNRF92), and the EliteForsk Prize. ORW was supported by EPSRC grant EP/M027783/1, the ERC under the European Union's Horizon 2020 research and innovation programme (grant agreement No.\ 756444), and a Philip Leverhulme Prize from the Leverhulme Trust.

\section{Proof of Theorem \ref{mainthm:stability}}

We first explain how to deduce Theorem \ref{mainthm:stability} from Theorems \ref{thm:scale-euler-char} and \ref{thm:scale-euler-char2}.

Let $\lambda : \mathrm{Fr}(\epsilon^1 \oplus TW) \overset{\rho}\to \overline{\Theta} \overset{\overline{u}}\to \Lambda$ be a factorisation into an $n$-connected $\GL_{2n+1}(\bR)$-equivariant cofibration $\rho$ and a $n$-co-connected $\GL_{2n+1}(\bR)$-equivariant fibration $\overline{u}$, and as above we write $\Theta$ for the underlying $\GL_{2n}(\bR)$-space of $\overline{\Theta}$ and $u$ for the underlying $\GL_{2n}(\bR)$-equivariant map of $\overline{u}$. There is then a map
\begin{equation}\label{eq:CompMap}
\cM^\Lambda(W,\lambda) \lra (\Omega^\infty MT\Theta) \moddd \hAut(u)
\end{equation}
which by \cite[Corollary 1.9]{GR-W4} is an isomorphism on $i$th (co)homology onto the path-component which it hits, as long as $i \leq \tfrac{\overline{g}(W, \lambda)-3}{2}$. (Note that by considering a $\GL_{2n+1}(\bR)$-space $\Lambda$ rather than a $\GL_{2n}(\bR)$-space, the  tangential structure $\Theta$ is ``spherical'' by the discussion after \cite[Definition 3.2]{GRWguide}, and so the stability range is as claimed.) Here $\bar{g}(W, \lambda)$ is the \emph{stable $\Lambda$-genus} of $(W, \lambda)$, the largest $g \in \bN$ for which there exists $h \in \bN$ such that $W \# h (S^n \times S^n)$ is $\Lambda$-diffeomorphic to $W_0 \# (g+h) (S^n \times S^n)$ for some $(W_0,\lambda_0)$.

Let $(W_0, \lambda_0)$ be a manifold stably $\Lambda$-diffeomorphic to $(W,\lambda)$ and minimising the quantity $(-1)^n \chi(W_0)$. Such a manifold has stable $\Lambda$-genus zero and hence for large enough $h$ we must have that $W \# h (S^n \times S^n)$ is $\Lambda$-diffeomorphic to $W_0 \# (h + \overline{g}(W, \lambda)) (S^n \times S^n)$, so
$$\overline{g}(W, \lambda) = (-1)^n(\chi(W) - \chi(W_0))/2.$$
It follows that \eqref{eq:CompMap} is an isomorphism on $i$th (co)homology as long as 
$$(-1)^n \chi(W) \geq 4i + \left(6+ (-1)^n \chi(W_0)\right).$$

If $(W', \lambda')$ is stably $\Lambda$-diffeomorphic to $(W, \lambda)$ then the same analysis applies, and there is a map
\begin{equation}\label{eq:CompMap2}
\cM^\Lambda(W',\lambda') \lra (\Omega^\infty MT\Theta) \moddd \hAut(u)
\end{equation}
which is an isomorphism on $i$th (co)homology onto the path-component which it hits, as long as
$$(-1)^n \chi(W') \geq 4i + \left(6+ (-1)^n \chi(W_0)\right).$$

By assumption we may write
$$a \cdot \chi(W) = b \cdot \chi(W')$$
for integers $a$ and $b$ all of whose prime factors are invertible in $\mathrm{End}_\bZ(\Coeffs)$. Furthermore the two Euler characteristics have the same parity, as (de)stabilisation changes the Euler characteristic by $\pm 2$, so if either $a$ or $b$ is even then both $\chi(W)$ and $\chi(W')$ are even too.

By Theorems \ref{thm:scale-euler-char} and \ref{thm:scale-euler-char2}, writing $\psi^x = \psi^{x/2^{v_2(x)}} \circ (\psi^{2})^{v_2(x)}$, (after perhaps implicitly localising away from 2) there are maps
\begin{equation*}
\xymatrix{
\Omega^\infty_{\chi(W)}MT\Theta \ar[rrr]^-{\psi^a} &&& \Omega^\infty_{a \chi(W)}MT\Theta \ar@{=}[d]\\
\Omega^\infty_{\chi(W')}MT\Theta \ar[rrr]^-{\psi^b} &&& \Omega^\infty_{b \chi(W')}MT\Theta
}
\end{equation*}
which are $\hAut(\overline{\Theta})$-equivariant and induce isomorphisms on $\Coeffs$-homology, as $\Coeffs$ is a $\bZ[a^{-1}, b^{-1}]$-module. By construction these maps do not change the $\pi_{-1} MT\overline{\Theta}$-component: we now analyse the components corresponding to $W$ and $W'$.

We now claim that $\psi^a([W, \rho]) = \psi^b([W', \rho']) \in \pi_0(\Omega^\infty MT\Theta)$ for a suitable choice of $\rho': \Fr(\epsilon^1 \oplus TW') \to \overline{\Theta}$ lifting $\lambda'$.  Since these two elements of $\pi_0(MT\Theta)$ have the same Euler characteristic, it suffices to arrange that they also have the same $\pi_{-1} MT\overline{\Theta}$-component.  The stable $\Lambda$-diffeomorphism from $(W,\lambda)$ to $(W', \lambda')$ gives a $\Lambda$-cobordism
$$X : W \# g (S^n \times S^n) \leadsto W' \# g' (S^n \times S^n)$$
which is furthermore an $h$-cobordism. We can therefore extend the $\overline{\Theta}$-structure given by $(W, \rho)$, stabilised, to a $\overline{\Theta}$-structure on $X$ lifting the given $\Lambda$-structure, and hence obtain a $\overline{\Theta}$-manifold $(W' \# g' (S^n \times S^n), \rho'')$ whose underlying $\Lambda$-manifold $(W' \# g' (S^n \times S^n), u \circ \rho'')$ is the stabilisation of $(W', \lambda')$. Now the $\overline{\Theta}$-manifolds
\begin{equation}\label{eq:NonDiffStab}
(W' \# g' (S^n \times S^n), \rho'') \text{ and } (W', \rho') \# g' (S^n \times S^n)
\end{equation}
need not be $\overline{\Theta}$-diffeomorphic, but must differ by an equivalence $f : \overline{\Theta} \to \overline{\Theta}$ over $\Lambda$ (see \cite[Lemma 9.2]{GR-W4}). However the $\overline{\Theta}$-structure $\rho'$ on $W'$ is merely a choice of lift of $\lambda'$ along $\overline{u}$, and by re-choosing it to be $f \circ \rho'$ we may then suppose that the manifolds \eqref{eq:NonDiffStab} are indeed $\overline{\Theta}$-diffeomorphic. With this choice we therefore have the desired
$$[W, \rho] = [W', \rho'] \in \pi_{-1} MT\overline{\Theta},$$
using the $\overline{\Theta}$-cobordism $X$ and the fact that this cobordism theory is insensitive to stabilisation by standard $S^n \times S^n$'s.  

Denoting by $[[W, \lambda]] \subset \pi_0 MT\Theta$ the $\pi_0 \hAut(\overline{u})$-orbit of $[W, \rho]$, and similarly $[[W', \lambda']]$, and using the forgetful homomorphism $\hAut(\overline{u}) \to \hAut(\overline{\Theta})$ to let the monoid $\hAut(\overline{u})$ act on $\Omega^\infty MT\Theta$, we therefore have a zig-zag of maps
\begin{equation}\label{eq:ZigZag}
\left(\Omega^\infty_{[[W, \lambda]]}MT\Theta\right) \hcoker \hAut(\overline{u}) \lra \cdot \longleftarrow \left(\Omega^\infty_{[[W', \lambda']]}MT\Theta\right) \hcoker \hAut(\overline{u})
\end{equation}
which induce isomorphisms on homology with $\Coeffs$-coefficients. The argument is completed by the following lemma.

\begin{lemma}
The natural map $\hAut(\overline{u}) \to \hAut(u)$ is a weak equivalence.
\end{lemma}
\begin{proof}
Working in the categories of $\GL_{2n}(\bR)$-spaces over $\Lambda$, or $\GL_{2n+1}(\bR)$-spaces over $\Lambda$,  we have
$$\mathrm{map}^{/\Lambda}_{\GL_{2n}(\bR)}(\Theta, \Theta) = \mathrm{map}^{/\Lambda}_{\GL_{2n+1}(\bR)}(\GL_{2n+1}(\bR) \times_{\GL_{2n}(\bR)}\Theta, \overline{\Theta})$$
but the natural $\GL_{2n+1}(\bR)$-equivariant map $\GL_{2n+1}(\bR) \times_{\GL_{2n}(\bR)}\Theta \to \overline{\Theta}$ has homotopy fibre $\GL_{2n+1}(\bR)/\GL_{2n}(\bR) \simeq S^{2n}$ so is $2n$-connected, whereas $\overline{u} : \overline{\Theta} \to \Lambda$ is $n$-co-connected, so the restriction map
$$\mathrm{map}^{/\Lambda}_{\GL_{2n+1}(\bR)}(\overline{\Theta}, \overline{\Theta}) \lra \mathrm{map}^{/\Lambda}_{\GL_{2n+1}(\bR)}(\GL_{2n+1}(\bR) \times_{\GL_{2n}(\bR)}\Theta, \overline{\Theta})$$
is an equivalence. The claim now follows by restricting to the path-components of homotopy equivalences.
\end{proof}

\begin{remark}
This argument also gives a conclusion about homology with certain local coefficients. The maps \eqref{eq:CompMap} and \eqref{eq:CompMap2} are in fact acyclic in a range of degrees \cite[Corollary 1.9]{GR-W4}, and the maps $\psi^q$ are acyclic with $\bZ[q^{-1}]$-module coefficients (as they are infinite loop maps which are isomorphisms on homology with these coefficients) so remain so after taking homotopy orbits by $\hAut(\overline{u})$.

So if $\mathcal{\Coeffs}$ is a system of local coefficients on the middle space of the zig-zag \eqref{eq:ZigZag}, with typical fibre $\Coeffs$ and having $v_p(\chi(W))=v_p(\chi(W'))$ for all primes $p$ which are not invertible in $\mathrm{End}_\bZ(\Coeffs)$, then there is also an isomorphism $H^i(\cM^\Lambda(W,\lambda);\mathcal{\Coeffs}) \cong H^i(\cM^\Lambda(W',\lambda');\mathcal{\Coeffs})$ in a range of degrees.
\end{remark}

\section{Proof of Theorems \ref{thm:scale-euler-char} and \ref{thm:scale-euler-char2}}

The proof of Theorem \ref{thm:scale-euler-char} is by an explicit construction of $\psi^q$ as a map of spectra.  The main ingredient is a certain commutative diagram of spectra, which we first describe informally. It is
\begin{equation*}
  \xymatrix{
    \Sigma^\infty \overline{B}_+ \ar@{=}[d] \ar[r]^p & {MT\Theta}
    \ar[r] \ar[d]^{sz} & S^1 \wedge MT\overline{\Theta}\ar[d]\\
    \Sigma^\infty \overline{B}_+ \ar[r]^{st} & \Sigma^\infty
    \overline{B}_+    \ar[r]  & C_{st}
  }
\end{equation*}
where $s: B \to \overline{B}$ is the natural map of Borel constructions.  The map $s$ is homotopy equivalent to a smooth fibre bundle with fibres $S^{2n}$ so we have a Becker--Gottlieb transfer $t:
\Sigma^\infty \overline{B}_+ \to \Sigma^\infty B_+$, factoring as a
pre-transfer $p: \Sigma^\infty \overline{B}_+ \to MT\Theta$ composed
with a map $z: MT\Theta \to \Sigma^\infty B_+$ induced by the zero
section of $\theta$.  The spectrum $C_{st}$ is defined to be the
homotopy cofibre of $st$, and both rows are cofibre sequences.  It
follows that the right square in the diagram is a homotopy
pullback, and hence we get the homotopy pullback diagram of infinite loop spaces~\eqref{eq:14} mentioned in the introduction.  On spectrum homology the map $st$ induces multiplication by $\chi(S^{2n}) = 2$, from which it follows that the homology and hence homotopy groups of $C_{st}$ are 2-power torsion. The space $\overline{B}$ is path connected, because $W$ is, so $\pi_0(\Sigma^\infty \overline{B}_+) = H_0(\Sigma^\infty \overline{B}_+;\bZ) = \bZ$. Thus $\pi_0(C_{st}) = \bZ/2$, and the map $\Sigma^\infty \overline{B}_+ \to C_{st}$ is surjective on $\pi_1$ because $st$ is injective on $\pi_0$.

To produce an endomorphism of $\Omega^\infty MT\Theta$
satisfying part~(\ref{item:1}) of the theorem, it therefore suffices to
produce an endomorphism of $\Sigma^\infty\overline{B}_+$ over
$C_{st}$.  For $q = 1 + 2k$, we may use the map
$\mathrm{id} + kst: \Sigma^\infty \overline{B}_+ \to \Sigma^\infty
\overline{B}_+$ which is obviously over $C_{st}$, at least in the
homotopy category, since $C_{st}$ is the cofibre of the map $st$.  In
spectrum homology, $st$ multiplies by $\chi(S^{2n}) = 2$ and hence
$\mathrm{id} + kst$ induces multiplication by $1 + 2k = q$ on $H_*(\Sigma^\infty \overline{B}_+;\bZ)$
ensuring part~(\ref{item:2}) of the theorem. Furthermore it acts by multiplication by $q$ on $\pi_0\Sigma^\infty \overline{B}_+ = \pi_0 Q(\overline{B}_+) = \bZ$, so indeed sends $\Omega^\infty_\chi MT\Theta$ to $\Omega^\infty_{q\chi} MT\Theta$.

It remains to explain how to achieve part~(\ref{item:3}) of the
theorem, that the continuous action of the topological monoid
$\hAut(\overline{\Theta})$ on the space $\Omega^\infty MT\Theta$ 
commutes with $\psi^q$.  It is not sufficient that $\psi^q$ commutes
up to homotopy with the action of individual elements of
$\hAut(\overline{\Theta})$, since we want to descend $\psi^q$ to the
homotopy orbit space.  To give a convincing proof, it seems best to
spell out a point-set model for the square~\eqref{eq:14}.

\begin{proof}[Proof of Theorem~\ref{thm:scale-euler-char}]
  As explained above, it remains to give a point-set model for the
  diagram~\eqref{eq:14} and the self-map $\mathrm{id} + kst$ of
  $Q(\overline{B}_+)$ over $\Omega^\infty C_{st}$, all of which
  commutes strictly with the action of $\hAut(\overline{\Theta})$.

We must adopt some conventions.  Let us consider $\GL_{2n}(\bR)$ as lying inside $\GL_{2n+1}(\bR)$ using the last $2n$ coordinates. Let us consider $\R^{N-1}$ as lying inside $\R^N$ as the subspace of vectors whose last coordinate is 0, and take $\R^\infty$ to be the direct limit. To form the Borel constructions we shall take $E\GL_{2n}(\bR) := \mathrm{Fr}_{2n}(\bR^\infty)$, and similarly take $E\GL_{2n+1}(\bR) := \mathrm{Fr}_{2n+1}(\bR \oplus \bR^\infty)$. The map $\mathrm{Fr}_{2n}(\bR^\infty) \to \mathrm{Fr}_{2n+1}(\bR \oplus \bR^\infty)$ which adds the basis vector of the first $\bR$-summand as the first element of the $(2n+1)$-frame is then equivariant for the inclusion $\GL_{2n}(\bR) \subset \GL_{2n+1}(\bR)$.

Then we have $B\GL_{2n+1}(\bR) = \Gr_{2n+1}(\bR \oplus \R^\infty)$, which we may filter in the usual way by $\Gr_{2n+1}(\bR \oplus \R^{N-1})$. Pulling back this filtration along the map $\overline{\theta} : \overline{B} \to \Gr_{2n+1}(\R^\infty)$, we set
$\overline{B}_{N} := (\overline{\theta})^{-1}(\Gr_{2n+1}(\R \oplus \R^{N-1}))$. There is an induced map $\overline{\theta}_N: \overline{B}_N \to \Gr_{2n+1}(\R \oplus \R^{N-1})$ and we shall write $\overline{\theta}_N^* \gamma^\perp = \overline{\theta}_N^* \gamma_{2n+1,N}^\perp$ for the pullback of the
  $(N - 2n -1)$-dimensional bundle of orthogonal complements. Then
  $MT\overline{\Theta}$ is the spectrum with $N$th space given by the Thom space
  $(\overline{B}_N) ^{\overline{\theta}_N^* \gamma^\perp}$, so that
  \begin{equation*}
    \Omega^{\infty-1} MT\overline{\Theta} = \colim_{N \to \infty}
    \Omega^{N-1} (\overline{B}_N)^{\overline{\theta}_N^* \gamma^\perp}.
  \end{equation*}
We similarly define $\theta_N : B_N \to \Gr_{2n}(\R^{N})$, and hence the spectrum $MT\Theta$. There is a map 
\begin{equation}\label{eq:jkgh}
\Gr_{2n}(\R^{N-1}) \hookrightarrow \Gr_{2n+1}(\R \oplus \R^{N-1}),
\end{equation}
given by direct sum with the 1-dimensional vector space given by the first $\bR$-summand, which induces a map $B_{N-1} \to \overline{B}_N$. The map \eqref{eq:jkgh} is
  $2n$-connected, but is covered by an $(N-2)$-connected map
  $\Gr_{2n}(\R^{N-1}) \to S(\gamma_{2n+1,N})$ and hence gives a $(N-2)$-connected map $B_{N-1} \to S(\overline{\theta}_N^*\gamma_{2n+1,N})$. Passing to Thom spaces this gives a $(2N - 2n -
  2)$-connected map
  \begin{equation*}
    S^1 \wedge (B_{N-1})^{(\overline{\theta}_N\vert_{B_{N-1}})^* \gamma_{2n,N-1}^\perp} \lra S(\overline{\theta}_N^*\gamma_{2n+1,N})^{\overline{\theta}_N^* \gamma_{2n+1, N}^\perp}.
  \end{equation*}
  These combine to define a map from $MT\Theta$ to the spectrum whose
  $(N-1)$st space is $S(\overline{\theta}_N^*\gamma_{2n+1,N})^{\overline{\theta}_N^* \gamma_{2n+1, N}^\perp}$, and this map is a weak equivalence.  This map is also $\hAut(\overline{\Theta})$-equivariant.  (This weak equivalence does not come  with a spectrum map in the other direction, let alone an equivariant one.)

  The square~\eqref{eq:14} will be assembled from a square of spaces
  fibred over $\overline{B}_N$, and we first explain the constructions
  on fibres.  Let $V \in \Gr_{2n+1}(\R^N)$ and write $S(V)$ for the
  unit sphere of $V$ and $S^V$ for the one-point compactification.  If
  $x \in \R^N$ we shall write $\pi_V(x) \in V$ for the orthogonal projection.  If
  $x \in V \setminus 0$ we shall write $\pi_S(x) = x/|x| \in S(V)$ for
  the nearest point in the sphere.  We will describe certain explicit
  maps $p(V): S^V \to S(V)^{\epsilon^1}$ and $z(V): S(V)^{\epsilon^1}
  \to S(V)_+ \wedge S^V$, and explain how the composition $z(V) \circ
  p(V)$ gives rise to a model for the Becker--Gottlieb transfer for a
  linear sphere bundle (indeed, we will just unwrap the definition of \cite[Section 3]{BG} in this case).

  The map
  \begin{equation*}
    p(V): S^V  \lra S(V)^{\epsilon^1},
  \end{equation*}
  is induced by the Pontryagin--Thom construction applied to the
  embedding $S(V) \subset V$.  In formulas, we can take e.g.\
  \begin{equation*}
    p(V)(x) = (\pi_S(x), \log |x|)
    \in S(V)_+ \wedge S^1 = S(V)^{\epsilon^1}
  \end{equation*}
  when $x \neq 0, \infty \in S^V$.  The Thom space $S(V)^{\epsilon^1}$
  is homeomorphic to the quotient $S^V/S^0$, and under this
  identification the map $p(V)$ is the quotient map.
  
  The map
  \begin{equation*}
    z(V): S(V)^{\epsilon^1} \lra S(V)^{TS(V)\oplus \epsilon^1} = S(V)_+
    \wedge S^V
  \end{equation*}
  is given by the zero section of the tangent bundle of $S(V)$.  In
  formulas, it sends $(x,t) \in S(V) \times \R \subset
  S(V)^{\epsilon^1}$ to $(x,tx) \in S(V) \times V \subset S(V)_+
  \wedge S^V$.

  If we compose these two maps and smash with $S^{V^\perp}$, we get
  \begin{equation*}
    S^N = S^V \wedge S^{V^\perp} \xrightarrow{p(V) \wedge \mathrm{id}}
    S(V)^{\epsilon^1} \wedge S^{V^\perp} \xrightarrow{z(V) \wedge
      \mathrm{id}} S(V)_+ \wedge S^V \wedge S^{V^\perp} = S(V)_+ \wedge S^N.
  \end{equation*}
  Finally, we write $s(V): S(V)_+ \wedge S^N \to S^N$ for the map
  induced by collapsing $S(V)$ to a point.  Then the composition
  \begin{equation*}
    b(V) = s(V) \circ (z(V) \wedge \mathrm{id}) \circ (p(V) \wedge
    \mathrm{id}): S^N \lra S^N
  \end{equation*}
  is a continuous map of degree $\chi(S^{2n}) = 2$ (by the Poincar{\'e}--Hopf theorem, see \cite[Theorem 2.4]{BG}), depending
  continuously on the point $V \in \Gr_{2n+1}(\R^N)$.  The resulting continuous
  map $b: \Gr_{2n+1}(\R^N) \to \Omega^N S^N$ in the limit gives a map
  $B\GL_{2n+1}(\bR) \to QS^0$ which is a model for the Becker--Gottlieb
  transfer of the sphere bundle over $B\GL_{2n+1}(\bR) \simeq B\mathrm{O}(2n+1)$.

  Now consider the diagram
  \begin{equation*}
    \xymatrix{
      S^N \ar[r]^-{p(V)} \ar@{=}[d] & {S(V)^{\epsilon^1} \wedge S^{V^\perp}}
      \ar[r]\ar[d]^{sz} & C_{p(V)}\ar[d]\\
      S^N \ar[r]^{st(V)} & S^N
      \ar[r] & C_{st(V)},
    }
  \end{equation*}
  where the entries in the right column are the mapping cylinders.
  Since $p(V)$ induces a homeomorphism $S^N/S^{V^\perp} \to
  S(V)^\epsilon \wedge S^{V^\perp}$, it follows from the Puppe sequence
  that there is a canonical induced homeomorphism $C_{p(V)} \cong S^1 \wedge
  S^{V^\perp}$.  Since $st(V)$ has degree 2, there is a homotopy
  equivalence from $C_{st(V)}$ to a mod 2 Moore space, but this is not
  quite sufficiently canonical for our purposes (since we get a
  different mod 2 Moore space for each $V$).  We have proved that for
  each $V \in \Gr_{2n+1}(\R^N)$ there is a canonical commutative diagram
  \begin{equation}\label{eq:7}
    \begin{aligned}
      \xymatrix{
        {S(V)^{\epsilon^1} \wedge S^{V^\perp}}
        \ar[r]\ar[d]^{sz} & S^1 \wedge S^{V^\perp}\ar[d]\\
        S^N
        \ar[r] & C_{st(V)},
      }  
    \end{aligned}
  \end{equation}
  which is a pushout and homotopy pushout.
  
  There is a canonical homotopy from the composition of $st(V): S^N
  \to S^N$ and $S^N \to C_{st(V)}$ to the constant map.  Suspending
  once, $S^1 \wedge S^N \to S^1 \wedge S^N \to S^1 \wedge C_{st(V)}$
  is canonically null homotopic.  If $k \geq 0$ is an integer, we may
  use the $S^1$ coordinate to form the sum of the identity map $1: S^1 \wedge S^N
  \to S^1 \wedge S^N$ and $k$ copies of the map $st(V): S^1 \wedge S^N
  \to S^1 \wedge S^N$.  We obtain a diagram
  \begin{equation}\label{eq:8}
    \begin{aligned}
      \xymatrix{
        {S^1 \wedge S^N}\ar[d]_{1 + kst(V)} \ar[r] & {S^1 \wedge
          C_{st(V)}} \ar@{=}[d]\\
        {S^1 \wedge S^N} \ar[r] & {S^1 \wedge
          C_{st(V)},}
      }
    \end{aligned}
  \end{equation}
  which commutes up to a canonical homotopy.  (The canonical nullhomotopy of each $st$ gives a homotopy from $1 + kst$ to the sum of
  the identity map and $k$ copies of the constant map; this is in turn
  canonically homotopic to the identity map.)  The homotopy class of
  the map $1 + kst(V): S^N \to S^N$ is determined by its degree which
  is $2k+1$, but the actual map depends in a non-trivial way on $V \in
  \Gr_{2n+1}(\R^N)$.

  All spaces in the diagram ``vary continuously in $V$'', in the sense
  that they are fibres over $V$ of fibre bundles over
  $\Gr_{2n+1}(\R^N)$.  The commutative diagram~\eqref{eq:7} in the
  category of spaces over $\Gr_{2n+1}(\R^N)$ may be pulled back along
  $\overline{\theta}_N: \overline{B}_N \to \Gr_{2n+1}(\R^N)$ to form a diagram
  \begin{equation}\label{eq:9}
    \begin{aligned}
      \xymatrix{
        {S(\overline{\theta}_N^*\gamma)^{\epsilon^1 \oplus \overline{\theta}_N^* \gamma^\perp}}
        \ar[r]\ar[d]^{sz} & S^1 \wedge \overline{B}_N^{\overline{\theta}_N^* \gamma^\perp}\ar[d]\\
        S^N\wedge (\overline{B}_N)_+
        \ar[r] & C_{st}^{\overline{B}_N},
      }  
    \end{aligned}
  \end{equation}
  which is again a pushout and homotopy pushout, where $C_{st}^{\overline{B}_N}$ is the mapping cylinder of the map $S^N \wedge (\overline{B}_N)_+ \to S^N \wedge (\overline{B}_N)_+$ given on $(v,x) \in S^N \times \overline{B}_N$ by $st(v,x) = (st(f(x))v,x)$.

  Similarly, the diagrams~\eqref{eq:8} assemble over $V$ to a
  diagram
  \begin{equation}\label{eq:10}
    \begin{aligned}
      \xymatrix{
        {S^1 \wedge S^N\wedge (\overline{B}_N)_+}\ar[d]_{1 + kst} \ar[r] & {S^1 \wedge
          C_{st}^{\overline{B}_N}} \ar@{=}[d]\\
        {S^1 \wedge S^N\wedge (\overline{B}_N)_+} \ar[r] & {S^1 \wedge
          C_{st}^{\overline{B}_N},}
      }
    \end{aligned}
  \end{equation}
  which commutes up to a canonical homotopy.

  Applying $\Omega^{N+1} S^1 \wedge (-)$ to the diagram~\eqref{eq:9} and
  letting $N \to \infty$ we get a model for~\eqref{eq:14}.  The monoid
  $\hAut(\overline{\Theta})$ acts on the whole diagram~\eqref{eq:9}, since it acts on $\overline{B}_N$ over $\Gr_{2n+1}(\R^N)$.  This gives a weak equivalence from $\Omega^\infty
  MT\Theta$ to the homotopy pullback in~\eqref{eq:14}, which is also
  an $\hAut(\overline{\Theta})$ equivariant map.  The monoid
  $\hAut(\overline{\Theta})$ also acts on the diagram~\eqref{eq:10},
  including the homotopy, and after applying $\Omega^{N+1}$ and taking
  $N \to \infty$ we obtain a self-map of $Q(\overline{B}_+)$ which is
  over $\Omega^{\infty} C_{st}$ up to a specified homotopy.  Again
  this self-map and the specified homotopy commutes strictly with the
  action of $\hAut(\overline{\Theta})$ since both the map and the
  homotopy arose from fibrewise constructions over $\Gr_{2n+1}(\R^N)$.

  Finally, the self-map of $Q(\overline{B}_+)$ induces an
  $\hAut(\overline{\Theta})$-equivariant self-map of the homotopy
  pullback of $Q(\overline{B}_+) \rightarrow \Omega^\infty C_{st}
  \leftarrow \Omega^{\infty -1} MT\overline{\Theta}$, and we have seen
  that this pullback is weakly equivalent to $\Omega^\infty MT\Theta$
  by an $\hAut(\overline{\Theta})$-equivariant map.
\end{proof}

\begin{proof}[Proof of Theorem~\ref{thm:scale-euler-char2}]
We continue with the notation developed above. The spectrum homology of
  $C_{st}$ is all $2$-torsion, so the localisation $C_{st}[\frac12]$ as a spectrum is contractible. However, the localised space $(\Omega^\infty C_{st})[\frac12]$ is not
  contractible since it has two components.  Instead, there is a
  spectrum map $w_{2n}: C_{st} \to H\bF_2$ which becomes an
  isomorphism in homology of infinite loop spaces with coefficients in
  any $\bZ[\frac12]$-module.  Similarly, the map
  \begin{equation*}
    \Omega^{\infty} MT\Theta \lra Q(\overline{B}_+)
    \times_{\Omega^\infty H\bF_2} \Omega^{\infty -1} MT\overline{\Theta}
  \end{equation*}
  induces an isomorphism in homology with coefficients in any
  $\bZ[\frac12]$-module, and hence a weak equivalence of localized
  spaces.  The spectrum map $2: S^0 \to S^0$ induces a self-map of
  $Q(\overline{B}_+)$ commuting with the action of $\hAut(\overline{\Theta})$ and whose restriction to the even-degree path components commutes with the map to $\Omega^\infty H\bF_2$.  This self-map can be used in place of $1 + kst$ to produce $\psi^2$.
\end{proof}

\section{An example}\label{sec:example}

In this section we will give an example to show that in Theorem \ref{mainthm:stability} it is indeed necessary to take homology with certain primes inverted. We will take as an example the 6-manifolds $V_d$ given by a smooth degree $d$ hypersurface in $\mathbb{CP}^4$, which we have studied in detail in \cite[Section 5.3]{GRWguide}. Any unattributed claims about these manifolds may be found there. We will also consider their stabilisations
$$V_{d,g} := V_d \# g(S^3 \times S^3)$$
obtained by connect-sum of $V_d$ with $g$ copies of $S^3 \times S^3$, which contain
$$g(V_{d,g}) = g + \tfrac{1}{2}(d^4 - 5d^3 + 10 d^2 - 10 d + 4)$$
copies of $S^3 \times S^3$.

\begin{theorem}\label{thm:ex}
Let $p \geq 7$ be a prime number, and suppose that $g(V_{d,g}) \geq 9$. Then
$$H^3(\cM^{\mathrm{or}}(V_{d,g});\bZ_{(p)}) \cong \bZ_{(p)}/\gcd(d, g).$$
\end{theorem}
The formula
$\chi = \chi(V_{d,g}) = d(10-10d+ 5d^2-d^3) -2g$
implies that $\gcd(d,g) = \gcd(d,\chi)$, so the theorem may also be written
\begin{equation*}
  H^3(\cM^{\mathrm{or}}(V_{d,g});\bZ_{(p)}) \cong \bZ/p^{\min(v_p(d),v_p(\chi))}\bZ.
\end{equation*}
Hence the moduli spaces for the oriented stably diffeomorphic manifolds $V_{d,g}$ and $V_{d,g'}$ have isomorphic $H^3(-;\bZ_{(p)})$ if and only if $v_p(\chi(V_{d,g})) = v_p(\chi(V_{d,g'}))$, provided those $p$-adic valuations are at most $v_p(d)$.

\begin{proof}[Proof of Theorem \ref{thm:ex}]
In \cite[Section 5.3]{GRWguide} we computed the $\bQ$-cohomology of $\cM^{\mathrm{or}}(V_{d,g})$ in a stable range. We will refer to details of the notation from that discussion, which differs slightly from the notation used earlier in this note.

Firstly, the $\bQ$-cohomology calculation goes through without significant changes for $\cM^{\mathrm{or}}(V_{d,g})$, because $V_{d,g}$ and $V_{d}$ have the same Moore--Postnikov $3$-stage, and because any orientation preserving diffeomorphism of $V_{d,g}$ must also act trivially on $H^2(V_{d,g};\bZ)$. The only difference is that the formula for the $d_3$-differential now involves characteristic numbers of $V_{d,g}$, which can be calculated to give
\begin{align*}
d_3(\kappa_{p_2}) &= 0\\
d_3(\kappa_{p_1^2}) &= 0\\
d_3(\kappa_{te}) &=\kappa_e=\chi(V_{d,g}) = d(10-10d+ 5d^2-d^3)-2g\\
d_3(\kappa_{t^2 p_1}) &= 2\kappa_{tp_1}= 2d(5-d^2)\\
d_3(\kappa_{t^4}) &= 4\kappa_{t_3} = 4d.
\end{align*}

Secondly, the $\bQ$-cohomology calculation yields an analogous $\bZ_{(p)}$-cohomology calculation for large enough primes $p$. Specifically the spectrum $MT\theta_d$ is $(-6)$-connected, so by the Atiyah--Hirzebruch spectral sequence the Hurewicz map 
$$\pi_i(MT\theta_d)_{(p)} \lra H_i(MT\theta_d; \bZ_{(p)}) \cong H_{i+6}(B_d ; \bZ_{(p)})$$
is an isomorphism as long as $i < 2p-3 -6 $, so as long as $i \leq 5$ since we have assumed that $p \geq 7$. As $p$ is odd we have
$$H^{*}(B_d ; \bZ_{(p)}) = H^{*}(B\mathrm{SO}(6) \times K(\bZ,2) ; \bZ_{(p)}) = \bZ_{(p)}[p_1, p_2, e, t].$$
Thus we have $\pi_1(\Omega^\infty_0 MT\theta_d)_{(p)}=0$, $\pi_2(\Omega^\infty_0 MT\theta_d)_{(p)} \cong \bZ_{(p)}^{5}$ with the isomorphism given by the tautological classes $\kappa_{p_2}, \kappa_{p_1^2}, \kappa_{te}, \kappa_{t^2 p_1}, \kappa_{t^4}$, and $\pi_3(\Omega^\infty_0 MT\theta_d)_{(p)}=0$. Therefore
$$H^i(\cM^{\theta_d}(V_{d,g} , \ell_{V_{d,g}}) ; \bZ_{(p)}) = \begin{cases}
\bZ_{(p)} & i=0\\
0 & i=1\\
\bZ_{(p)}\{\kappa_{p_2}, \kappa_{p_1^2}, \kappa_{te}, \kappa_{t^2 p_1}, \kappa_{t^4}\} & i=2\\
0 & i=3.
\end{cases}$$

The submonoid $G \leq \mathrm{hAut}(u)$ of those path components which stabilise $[V_{d,g}, \ell_{V_{d,g}}]$ is path connected, and as the map $u: B_d \to B\mathrm{SO}(6) \times K(\bZ,2)$ is a $\bZ_{(p)}$-homology equivalence, since $p$ is odd, we also have that $\pi_i(G) \otimes \bZ_{(p)}=0$ for $i>0$. Thus the map $\cM^{\theta_d}(V_{d,g}, \ell_{V_{d,g}}) \to \cM^{\mu}(V_{d,g}, u \circ \ell_{V_{d,g}})$ is a $\bZ_{(p)}$-homology equivalence.

It remains to study the Serre spectral sequence for the fibration sequence
$$\cM^{\mu}(V_{d,g}, u \circ \ell_{V_{d,g}}) \lra \cM^{\mathrm{or}}(V_{d,g}) \lra K(\bZ,3),$$
which in low degrees has a single differential
$$d_3 : E_3^{0,2} = \bZ_{(p)}\{\kappa_{p_2}, \kappa_{p_1^2}, \kappa_{te}, \kappa_{t^2 p_1}, \kappa_{t^4}\} \lra E_3^{3,0} = H^3(K(\bZ,3);\bZ_{(p)})=\bZ_{(p)}$$
given by the formula above, so $H^3(\cM^{\mathrm{or}}(V_{d,g});\bZ_{(p)})$ is given by the cokernel of this differential. The claim now follows by the identity of ideals
$$(4d, 2d(5-d^2),d(10-10d+ 5d^2-d^3)-2g) = (d,g)$$
of $\bZ_{(p)}$, using again that $p$ is odd.
\end{proof}

\bibliographystyle{amsalpha}
\bibliography{biblio}

\end{document}